\documentclass[10pt]{amsart}
\usepackage{amsmath}
\usepackage{amsfonts}
\usepackage{amsthm}
\usepackage{amssymb}
\usepackage{graphicx}
\newtheorem{theorem}{Theorem}
\newtheorem*{theorem*}{Theorem}
\newtheorem*{acknowledgement*}{Acknowledgement}
\newtheorem*{definition*}{Definition}

\newtheorem{lemma}[theorem]{Lemma}

\newtheorem{proposition}[theorem]{Proposition}
\newtheorem{remark}[theorem]{Remark}

\newcommand{\RR}[0]{\mathbb{R}}

\newcommand{\CC}[0]{\mathbb{C}}
\newcommand{\pd}[2]{\frac{\partial #1}{\partial#2}}

\newcommand{\ve}[1]{\mathbf{#1}}

\newcommand{\Rc}[0]{\operatorname{Rc}}
\newcommand{\Rm}[0]{\operatorname{R}}

\newcommand{\dfn}[0]{\doteqdot}

\newcommand{\pdtau}[0]{\pd{}{\tau}}

\newcommand{\abs}[1]{\left\vert#1\right\vert}

\newcommand{\lam}{\lambda}

\newcommand{\rad}{R}

\newcommand{\Ec}{E}

\newcommand{\tf}[0]{\wedge^2 T^*M}
\newcommand{\Pm}[0]{\operatorname{P}}
\newcommand{\Pmt}[0]{\tilde{\Pm}}
\newcommand{\Pmh}[0]{\hat{\Pm}}
\newcommand{\Sm}[0]{\operatorname{S}}
\newcommand{\Fm}[0]{\operatorname{F}}
\newcommand{\Rmt}[0]{\tilde{\Rm}}
\newcommand{\Rmh}[0]{\hat{\Rm}}
\newcommand{\Smt}[0]{\tilde{\Sm}}
\newcommand{\Smh}[0]{\hat{\Sm}}
\newcommand{\Am}[0]{\operatorname{A}}
\newcommand{\Bm}[0]{\operatorname{B}}
\newcommand{\Gm}[0]{\operatorname{G}}
\newcommand{\End}[0]{\operatorname{End}}

\newcommand{\Hc}[0]{\mathcal{H}}
\newcommand{\Kc}[0]{\mathcal{K}}

\title[K\"ahlerity of shrinkers asymptotic to K\"ahler cones]{K\"ahlerity of shrinking gradient Ricci solitons 
asymptotic to K\"ahler cones}

\author{Brett Kotschwar}
\email{kotschwar@asu.edu}
\address{School of Mathematical and Statistical Sciences,
	Arizona State University, Tempe, AZ 85287, USA}
	
 \date{November 2016}

\thanks{The author was partially supported by Simons Foundation grant \#359335.}

\begin{document}
\begin{abstract}
We prove that a shrinking gradient Ricci soliton which is asymptotic to a K\"ahler cone along some end is itself K\"ahler
on some neighborhood of infinity of that end. When
the shrinker is complete, it is globally K\"ahler. 
\end{abstract}
\maketitle

\section{Introduction} 
Let $M$ be a smooth manifold of dimension $n$. A \emph{shrinking gradient Ricci soliton} structure on $M$ consists of a Riemannian metric $g$
and a smooth function $f$ which together satisfy the equations
\begin{equation}\label{eq:shrinker}
  \Rc(g) + \nabla\nabla f= \frac{1}{2} g, \quad R + |\nabla f|^2 = f.
\end{equation}
The latter equation is a normalization which can be achieved by adding an appropriate constant to $f$
on any connected component of $M$.

Shrinking gradient Ricci solitons (or \emph{shrinkers}) correspond naturally to shrinking self-similar solutions to the Ricci flow, which are both the generalized fixed points
of the equation and the prototypical
models for the behavior of a solution 
in the vicinity of a developing singularity. They have been fully classified in dimensions two \cite{HamiltonSingularities} and three 
\cite{CaoChenZhu}, \cite{HamiltonSingularities}, \cite{Ivey3DSoliton}, \cite{Naber4D}, \cite{NiWallach}, \cite{Perelman2}, \cite{PetersenWylie},
and even a limited extension of this classification to dimensions four and higher would 
help advance the understanding of the long-time behavior of solutions to the equation.

There are indications that the asymptotic behavior of complete noncompact shrinking solitons may be rigid enough to support a 
classification into general structural types. At present, all known examples of complete noncompact shrinkers which are not locally reducible as products
are connected at infinity and smoothly asymptotic to regular cones (see \cite{FeldmanIlmanenKnopf},  \cite{DancerWang}, \cite{Yang}). Recent
results of Munteanu-Wang \cite{MunteanuWangKaehler, MunteanuWangConical, MunteanuWangGRSStructure} suggest that these two types may actually exhaust the possibilities for complete noncompact four-dimensional shrinkers.

In this paper, we consider shrinkers of the second type, whose geometries are asymptotically conical in a sense which we now make precise.
Given a closed Riemannian manifold $(\Sigma, g_{\Sigma})$ of (real) dimension $n-1$ and a fixed $a \in \RR$, let
$E_a = E_a(\Sigma)$ denote the cylinder $(a, \infty) \times \Sigma$
and $g_c = dr^2 + r^2g_{\Sigma}$
the regular conical metric on $E_0$.  Further, for any fixed $\lambda > 0$, let 
$\rho_\lam:E_0\to E_0$ denote the dilation map $\rho_\lam(r,\sigma)\dfn (\lam r, \sigma)$. By an \emph{end} $V\subset (M, g)$, we will mean an unbounded connected component of $M\setminus D$ for some compact $D\subset M$.
\begin{definition*}
\label{def:asymcone} Let $V$ be an end of $(M, g)$.
We say that $(M, g)$ is \emph{asymptotic to $(E_0,g_c)$ along $V$}
if, for some $r > 0$, there is a diffeomorphism $\Phi:\Ec_r\to V$ such that $\lam^{-2}\rho_\lam^*\Phi^\ast g\to g_c$
in $C^2_{\emph{loc}}(E_0, g_c)$ as $\lam \to \infty$.
\end{definition*}

Using the correspondence between shrinkers and self-similar solutions, it is possible to show that any shrinker whose curvature decays quadratically
on an end will be asymptotic to some regular cone on that end in the above sense and that the convergence 
is actually locally smooth
(see, e.g., \cite{KotschwarWangConical} and cf. \cite{ChowLu}). 
Munteanu and Wang \cite{MunteanuWangConical} have shown that any complete shrinker whose Ricci curvature tends to zero at infinity must have quadratic curvature decay
and must therefore be asymptotically conical along each of its ends. In dimension four,
 they have proven that the same is true assuming only that the scalar curvature tends to zero \cite{MunteanuWang4D}. 

The uniqueness theorem proven in \cite{KotschwarWangConical} asserts that if two shrinking gradient Ricci solitons are asymptotic
to the same regular cone along some end of each, then the solitons are themselves isometric on some neighborhood of infinity of these ends. One consequence
of this theorem is that any non-trivial isometry of the cross-section of the cone must be reflected in a non-trivial isometry of the end of the soliton. 
In this paper, we show that the K\"ahlerity of the cone is another feature which is necessarily inherited
by an asymptotic shrinker.

\begin{theorem}\label{thm:kaehlerity}
 Suppose $(M, g, f)$ is a shrinking gradient Ricci soliton asymptotic to the regular cone $(\Ec_0, g_c)$ along the end $V\subset M$.
 If $(\Ec_0, g_c)$ is K\"ahler with complex structure $J_c$, then there is an end $W\subset V$ of $M$
 and a complex structure $J$ defined on $W$ relative to which $g|_W$ is K\"ahler and $(W, J)$ is biholomorphic to $(\Ec_R, J_c)$ 
 for some $R > 0$. If $(M, g)$ is complete, the K\"ahler structure extends to all of $M$.
\end{theorem}

If it is already known that there is a K\"ahler shrinker $(\hat{M}, \hat{g}, \hat{f})$ asymptotic to $(\Ec_0, g_c)$ along some end of $(\hat{M}, \hat{g})$,
the first assertion follows from the uniqueness result in \cite{KotschwarWangConical}. Here, however, the existence of such a K\"ahler competitor $(\hat{M}, \hat{g}, \hat{f})$ is not presumed in advance
and the K\"ahlerity of the soliton $(M, g, f)$ on the end $W$ is established only 
from its asymptotic K\"ahlerity along $V$. The assumptions are entirely local to the end $V$; in particular, $(M, g)$ need not be complete nor have only one end. However, when $(M, g)$ is complete and connected, 
its universal cover $(\tilde{M}, \tilde{g})$ will be globally K\"ahler and, consequently, connected at infinity \cite{MunteanuWangKaehler}.
It follows then that $(M, g)$ is also connected at infinity and that the K\"ahler structure on $W$ must extend to all of $M$.

Our proof of Theorem \ref{thm:kaehlerity} uses a combination of the methods in \cite{KotschwarRFHolonomy} and \cite{KotschwarWangConical}. In \cite{KotschwarWangConical}, it is shown that a shrinker which is asymptotic
to a cone $(E_0, g_c)$ gives rise to a smooth solution $g(t)$ to the Ricci flow on some end $W\subset V$ for $t\in [-1, 0)$
which begins at $g(-1) = g$, evolves self-similarly for $-1 \leq t < 0$, and converges smoothly as $t\nearrow 0$ to a limit metric isometric to $g_c$. 
By this construction, recalled in Proposition \ref{prop:brf} below, the first assertion in Theorem \ref{thm:kaehlerity} is reduced to a problem of ``backward propagation
of K\"ahlerity,'' that is, to the problem of showing that
a solution to the Ricci flow which becomes K\"ahler after some time must have actually been K\"ahler all along. 
This problem was considered in \cite{KotschwarRFHolonomy} within the class of complete solutions of bounded curvature. 

Whereas 
the problem of the propagation of the K\"ahler structure forward in time can, in principle, be 
reduced to the uniqueness of solutions to the Ricci flow by the construction of K\"ahler solutions with the same initial data,
the problem of its propagation backward in time does not seem to admit a correspondingly straightforward reduction to the backward uniqueness of solutions. The 
K\"ahler solutions which one would need to construct to act as competitors are now the solutions to parabolic \emph{terminal-value} problems
which are ill-posed in general.
These obstacles are discussed in some detail in Section 2.3 of \cite{KotschwarRFHolonomy} in the context of the preservation of holonomy along the flow. 

Instead, the problem of backward propagation of restricted holonomy in \cite{KotschwarRFHolonomy} is recast as one of backward uniqueness for a prolonged system
encoding the components of the curvature operator relative to a certain decomposition of the bundle of two-forms on $M$. In the next section, we discuss the specialization of this formulation to the K\"ahler setting. This sets up 
a problem of backward uniqueness which, analytically, turns out to be identical in structure and setting to that considered in \cite{KotschwarWangConical},
and can be resolved by the same argument (the application of appropriate Carleman inequalities) given in that reference. We discuss the application of these results
to the first assertion of Theorem \ref{thm:kaehlerity} in Section 3, and prove the second assertion, concerning the extension of the K\"ahler structure to the entire 
manifold,
with a short synthetic argument in Section 4.

\section{Backward propagation of K\"ahlerity under the Ricci flow}
In this section, we will assume that $g_0$ is a K\"ahler metric on $M$ with complex structure $J_0$ and that
$g = g(\tau)$ is a smooth family of metrics satisfying the \emph{backward Ricci flow}
\begin{equation}\label{eq:brf}
 \pdtau g = 2\Rc(g) 
\end{equation}
on $M\times [0, T]$ with $g(0) = g_0$. We do not assume here that $g(\tau)$
is self-similar or that $(M, g(\tau))$ is complete.

Our aim is to set up a framework to determine when the K\"ahlerity of the metric $g_0$ is transferred to  $g(\tau)$ for $\tau > 0$.
As we have noted above, we cannot simply reduce the problem of the the propagation of the K\"ahler structure to that of the uniqueness of \emph{solutions}
to \eqref{eq:brf} (that is, to the problem of backward uniqueness of solutions to the Ricci flow). Instead we will specialize to the K\"ahler setting
the approach in \cite{KotschwarRFHolonomy}, which encodes the general problem of preservation of reduced holonomy in terms of the uniqueness of solutions to an associated system of mixed differential inequalities.

\subsection{A family of almost-complex structures on $M$.}
The first step in our reduction of the problem is to extend $J_0$ to a family of almost complex structures $J = J(\tau)\in \operatorname{End}(TM)$ for $\tau\in [0, T]$
relative to which $g(\tau)$ remains Hermitian. (In our application, we will eventually we will wish to argue that $J(\tau) = J_0$.)  We do so by solving the ordinary differential initial value problem
\begin{equation}\label{eq:jev}
  \left\{\begin{array}{rl}
	      \pdtau J^a_b =  R_b^c J^a_c - R_c^aJ_b^c & \mbox{on}\ M\times (0, T]\\
	      J(x, 0) = J_0(x) & \mbox{on}\ M
         \end{array}\right.
\end{equation}
in each fiber of $\operatorname{End}(TM)$. 

The evolution equation in \eqref{eq:jev} may be understood in terms of the operator $D_{\tau}$ which acts 
on families $V = V(\tau)$ of $(k, l)$-tensors by
\begin{align*}
 D_{\tau}V_{b_1b_2 \ldots b_k}^{a_1a_2\ldots a_l} &= \pdtau V_{b_1b_2 \ldots b_k}^{a_1a_2\ldots a_l} 
 - R_{b_1}^{c}V_{cb_2\ldots b_k}^{a_1a_2\ldots a_l}  - R_{b_2}^{c}V_{b_1c\ldots b_k}^{a_1a_2\ldots a_l}
 - \cdots - R_{b_k}^{c}V_{b_1b_2\ldots c}^{a_1a_2\ldots a_l}\\
 &\phantom{=\pdtau V_{b_1b_2 \ldots b_k}^{a_1a_2\ldots a_l} }+ R_{c}^{a_1}V_{b_1b_2\ldots b_k}^{ca_2\ldots a_l}
 + R_{c}^{a_2}V_{b_1b_2\ldots b_k}^{a_1c\ldots a_l} + \cdots + R_{c}^{a_l}V_{b_1b_2\ldots b_k}^{a_1a_2\ldots c}. 
\end{align*}
Relative to a smooth family of local frames $\{e_{i}(\tau)\}_{i=1}^{n}$ evolving so as to remain
orthonormal relative to $g(\tau$), the components of $D_{\tau}V$ express the total derivatives 
\[
  D_{\tau}V_{b_1b_2 \ldots b_k}^{a_1a_2\ldots a_l} = \pdtau \left(V(e_{b_1}, e_{b_2}, \ldots, e_{b_k}, e_{a_1}^*, e_{a_2}^*, \ldots, e_{a_l}^*)\right).
\]
Alternatively, $D_\tau$ may be regarded as a vector on the product of the frame bundle with $[0, T]$ tangent
to the bundle of $g(\tau)$-orthonormal frames. See \cite{HamiltonHarnack} or Appendix F of \cite{RFV2P2} for details.

Equations \eqref{eq:brf} and \eqref{eq:jev} are then equivalent to the assertions that $D_{\tau} g = 0$ and $D_{\tau} J =0$,
and imply that
\[
J^2  = -\operatorname{Id}, \quad g(\cdot, \cdot) = g(J\cdot, J\cdot),
\]
that is, that $J$ remains an almost complex structure and $g$ remains Hermitian with respect to $J$ on $M\times[0, T]$,

\subsection{A time-dependent splitting of $\tf$.}
We will now examine the relationships between several families of endomorphisms of $\tf$ induced by $J$. Below, we will write
\[
    \theta\wedge \sigma = \theta\otimes\sigma - \sigma\otimes \theta,
\]
for $\theta$, $\sigma\in T^*M$, and use the metric
\begin{equation}\label{eq:metric}
 \langle \theta\wedge\sigma, \phi\wedge \psi\rangle = \det\left(\begin{array}{cc}
                                                               \langle \theta, \phi \rangle & \langle \theta, \psi\rangle \\
							       \langle \sigma, \phi\rangle & \langle \sigma, \psi\rangle
                                                              \end{array}\right)
\end{equation}
induced on $\tf$ by $g(\tau)$.

Now define $\Fm$, $\Gm\in \End(\tf)$ by
\begin{align*}
 (\Fm\eta)(X, Y) &= \frac{1}{2}\left(\eta(X, JY) + \eta(JX, Y)\right), \quad (\Gm\eta)(X, Y) = \eta(JX, JY),
\end{align*}
for $\eta \in \tf$ and $X$, $Y$ in $TM$, and let 
\[
 \Pmt = \frac{1}{2}(\operatorname{Id} + \mathrm{G}), \quad \Pmh = \frac{1}{2}(\operatorname{Id} - \Gm). 
\]
The identities
\begin{equation}\label{eq:gfrel}
  \quad \Gm^2 = \operatorname{Id}, \quad \Fm \Gm = \Gm \Fm = - \Fm
\end{equation}
satisfied by $\Fm$ and $\Gm$ imply the following relations between $\Fm$, $\Pmt$, and $\Pmh$.

\begin{lemma}\label{lem:projprop}
On $M\times[0, T]$ we have the identities
\begin{align*}
 &\Pmt^2 = \Pmt, \quad \Pmh^2 = \Pmh, \quad \Pmh\Pmt = \Pmt\Pmh = 0, \quad \Pmt + \Pmh = \operatorname{Id},\\ 
 &\qquad\Pmt^* = \Pmt, \quad \Pmh^* = \Pmh, \quad \Fm^* = - \Fm,\\
 &\Fm^2 = - \Pmh, \quad \Pmt \Fm = \Fm\Pmt = 0,
 \quad \Fm\Pmh = \Pmh \Fm = \Fm.
\end{align*}
Here $\,^*$ denotes the adjoint of the operator relative to the metric \eqref{eq:metric}.	
\end{lemma}
The first two rows of relations imply, among other things, that 
\[
 \mathcal{H} = \mathcal{H}(\tau) = \operatorname{im}\Pmt(\tau), \quad \mathcal{K} = \mathcal{K}(\tau) = \operatorname{im} \Pmh(\tau),
\]
are complementary orthogonal subbundles of $\tf$.

\begin{lemma}\label{lem:splitting} For any $(p, \tau) \in M\times [0, T]$, we have the $g(\tau)$-orthogonal decomposition
\begin{equation}\label{eq:splitting}
    \wedge^2(T_p^*M) = \mathcal{H}_p(\tau)\oplus \mathcal{K}_p(\tau)
\end{equation}
and $\mathcal{H}_p(\tau)$ is a subalgebra of $\wedge^2(T_p^*M)$ isomorphic to $\mathfrak{u}(n/2)$ relative to the bracket
\begin{equation}\label{eq:bracket}
 [\omega, \eta]_{ij} = \omega_{ik}\eta_{kj} - \omega_{jk}\eta_{ki}
\end{equation}
induced by $g(\tau)$.
\end{lemma}
\begin{proof}
 As we have noted, the first assertion follows from the first two rows of Lemma \ref{lem:projprop}.
 For the second assertion, let $\mathfrak{s}_p(\tau)\cong \mathfrak{so}(n)$
 denote the Lie subalgebra of endomorphisms of $T_pM$ (under commutation) which are skew-symmetric relative to $g(p, \tau)$.
 The endomorphisms in $\mathfrak{s}_p(\tau)$ which further commute with $J = J(p, \tau)$ comprise a subalgebra $\mathfrak{u}_p(\tau)$
 of $\mathfrak{s}_p(\tau)$ isomorphic to $\mathfrak{u}(n/2)$. The map 
 $\Phi_{p, \tau}: \mathfrak{s}_p(\tau) \to \wedge^2(T^*_pM)$ given by $\Phi_{p, \tau}(A)(X, Y) =  g(X, AY)$
 is an isomorphism relative to the bracket \eqref{eq:bracket} on $\wedge^2(T^*_pM)$, and the image of $\mathfrak{u}_p(\tau)$
 under $\Phi$, consisting of those two-forms $\eta$ for which $\eta(X, Y) = \eta(JX, JX)$ for all $X$, $Y\in T_pM$, 
 is precisely $\mathcal{H}_p(\tau)$.
\end{proof}

\begin{remark} Relative to $\wedge_{\CC}^2T^*M = \tf\otimes_{\RR} \CC$, we have
\[
 \Pm^{(2, 0)} = \frac{1}{2}\left(\Pmh - \sqrt{-1}\Fm \right),\quad
 \Pm^{(0, 2)} = \frac{1}{2}\left(\Pmh + \sqrt{-1}\Fm \right),\quad
 \Pm^{(1, 1)} = \Pmt,
\]
where $\Pm^{(2, 0)}$, $\Pm^{(0, 2)}$, $\Pm^{(1, 1)}$ are the projections of $\wedge_{\CC}^2T^*M$
onto $\wedge^{2, 0}M$, $\wedge^{0, 2}M$, and $\wedge^{1, 1}M$, respectively.
\end{remark}

\subsection{The decomposition of the curvature operator}
Let us use $R$ to denote the $(4, 0)$ curvature tensor and $\Rm:\tf \to \tf$ to denote the curvature operator of $g(\tau)$. We adopt the convention
that
\[
      \Rm(\eta)_{ij} = -R_{ijab}\eta_{ab}, \quad \eta\in \tf.
\]
Relative to an orthonormal frame $\{e_i\}_{i=1}^n$, we have
\[
 \Rm_{ijkl} = \Rm_{klij} = \langle \Rm(e_i^*\wedge e_j^*), e^*_k\wedge e^*_l\rangle = -2R_{abij}\langle e^*_a\wedge e^*_b, e^*_k\wedge e_l^*\rangle
 = -2R_{ijkl}.
\]
Similarly, we will use $\Sm$ for the operator $\nabla \Rm\in  T^*M\otimes\End(\tf)$ such that
\[
 \Sm(X, \eta)_{ij} = -2\nabla_aR_{ijbc}X_a\eta_{bc}
\]
and adopt the notation
\[
 \Sm_{mijkl} = \langle\Sm(e_m, e_i^*\wedge e_j^*), e^*_k\wedge e^*_l\rangle = -2\nabla_mR_{ijkl}.
\]

Now define $\Rmt$, $\Rmh: \tf\to\tf$ by
\[
\Rmt = \Rm \circ \Pmt, \quad \Rmh = \Rm\circ\Pmh
\]
and $\Smt$, $\Smh\in T^*M\otimes\End(\tf)$ by
\[
\Smt = \nabla \Rm \circ \Pmt, \quad \Smh = \nabla\Rm\circ\Pmh.
\]
Here,
\[
 \Smt(X,\eta) = (\nabla_X\Rm)\circ \Pmt(\eta), \quad \Smh(X, \eta) = (\nabla_X\Rm) \circ \Pmh(\eta),
\]
for $X\in TM$ and $\eta\in \tf$.

\begin{lemma}\label{lem:init} At $\tau = 0$, we have
\begin{align}\label{eq:p0}
   & \nabla \Pmt = \nabla \Pmh = \nabla \Fm = 0,
\end{align}
\begin{align}
\label{eq:r0}
    &\Rm \circ \Fm = 0, \quad \Rmh = 0, \quad \Rmt = \Rm,
\end{align}
and
\begin{align}
\label{eq:dr0}
    &\nabla \Rm \circ \Fm = 0, \quad \Smh = 0, \quad \Smt = \nabla \Rm.
\end{align}
\end{lemma}
\begin{proof}
 When $\tau = 0$, we have $\nabla J = 0$ and consequently
 \[
  R(X, Y, JZ, W) + R(X, Y, Z, JW) = 0	
\]
for all $X$, $Y$, $Z$, $W\in TM$. The former implies that $\nabla \Fm = \nabla \Pmh = \nabla \Pmt = 0$ and the latter implies that
$\Rm\circ \Fm = 0$; in fact,
\[
 (\Rm\circ \Fm)(Z^*\wedge W^*)(X, Y) = R(X, Y, JZ, W) + R(X, Y, Z, JW).
\]
Since $\Pmh = -\Fm^2$ and $\Pmt = \operatorname{Id}- \Pmh$ by Lemma \ref{lem:projprop},
we also have $\Rm\circ \Pmh = 0$ and  $\Rm\circ \Pmt = \Rm$. Then
\[
  0 = \nabla (\Rm\circ \Fm) = \nabla \Rm \circ \Fm + \Rm \circ \nabla \Fm = \nabla \Rm \circ \Fm,
\]
and the other identities in \eqref{eq:dr0} follow similarly.
\end{proof}

 The identities in Lemma \ref{lem:init} reflect that the image of the curvature operator
 is contained in the parallel subalgebra $\mathfrak{hol}(g(0))$ of $\tf$ defined by the 
 holonomy representation of $g(0)$.
 Since $g(0)$ is K\"ahler, this means that, at any $p\in M$,
\[
 \operatorname{im}\Rm(p, 0) \subset\mathfrak{hol}_p(g(0)) \subset \mathcal{H}_p(0) \cong \mathfrak{u}(n/2),
\]
and so
\[
 \mathcal{K}_p(0) = \mathcal{H}_p^{\perp}(0) \subset \operatorname{ker} \Rm(p, 0), 
\]
since $\Rm(p, 0)$ is self-adjoint.

\subsection{A closed system of mixed differential inequalities}

Now define the sections
\[
\Am = \nabla \Pmh\in T^*M\otimes\End(\tf), \quad \Bm = \nabla\nabla \Pmh \in T^*M\otimes T^*M\otimes \End(\tf).
\]
It follows from the general computations in Section 4 of \cite{KotschwarRFHolonomy}
 that $\Rmh$ and $\Smh$, together with $\Am$ and $\Bm$, satisfy a closed system of differential inequalities.

\begin{proposition}\label{prop:sysev} The sections $\Am$, $\Bm$, $\Rmh$, $\Smh$ satisfy 
\begin{align}
\label{eq:aev}
      \left|D_{\tau} \Am \right| &\leq C|\Rm||\Am| + C|\Smh|\\  
\label{eq:bev}
     \left|D_{\tau}\Bm\right| &\leq C|\nabla \Rm| |\Am|+ C|\Rm|\left(|\Bm| + | \nabla \Smh|\right) \\
 \label{eq:rhev}
 \left|\left(D_{\tau} + \Delta\right)\Rmh\right| &\leq C|\nabla \Rm||\Am| + C|\Rm|\left(|\Bm| + |\Rmh|\right)\\
 \label{eq:shev}
 \left|\left(D_{\tau} + \Delta\right)\Smh\right| &\leq C|\nabla\nabla\Rm| |\Am| + C|\nabla\Rm||\Rmh| + C|\Rm|\left(|\Bm| + |\Rmh| + |\Smh|\right)
\end{align}
for some $C = C(n)$ on $M\times [0, T]$.
\end{proposition}
The only properties of the projections $\Pmt$ and $\Pmh$ of which the argument in \cite{KotschwarRFHolonomy} makes use is that they remain complementary and orthogonal,
evolve according
to $D_{\tau} \Pmt = D_{\tau}\Pmh = 0$ (which follows here from the fact that
$D_{\tau} J = 0$), and that $\Hc = \Pmt(\tf)$ remains closed under the Lie bracket. The computations actually hinge on the relations
\begin{equation}\label{eq:larelations}
    [\Hc, \Hc] \subset \Hc, \quad [\Hc, \Kc] \subset \Kc,
\end{equation}
implied by this latter condition. We will review here the derivation of \eqref{eq:rhev}, but 
refer the reader to \cite{KotschwarRFHolonomy} for the rest of the details. 

Recall that, under \eqref{eq:brf}, the curvature operator evolves according to the equation
\begin{equation}\label{eq:rmev}
  \left(D_{\tau} + \Delta\right)\Rm = - \Rm^2 - \Rm^{\#},
\end{equation}
where the second term denotes the Lie algebra square $\Rm^{\#} = \Rm \#\Rm$. Here, for 
$\mathrm{M}$, $\mathrm{N} \in \End(\tf)$, $\mathrm{M}\#\mathrm{N}\in \End(\tf)$ is the operator defined at $p\in M$ by
\begin{equation}\label{eq:sharpdef}
 (\mathrm{M}\# \mathrm{N})(\eta) = \frac{1}{2}\sum_{i, j} \left\langle [\mathrm{M}\varphi_i, \mathrm{N}\varphi_j], \eta\right\rangle [\varphi_i, \varphi_j],
\end{equation}
where $\{\varphi_i\}_{i=1}^{n(n-1)/2}$ is any orthonormal basis for $\wedge^2T_p^*M$.

Now, we may compute directly from \eqref{eq:rmev} that
\begin{align*}
 \left(D_{\tau} + \Delta\right)\Rmh &=  -(\Rm^2 + \Rm^{\#})\circ \Pmh + 2\nabla_i \Rm \circ \nabla_i \Pmh + \Rm \circ \Delta \Pmh,
\end{align*}
which immediately yields
\begin{align*}
 \left|\left(D_{\tau} + \Delta\right)\Rmh\right|&\leq C|\Rm|(|\Rmh| + |\Bm|) + C|\nabla \Rm| |\Am| + |\Rm^{\#}\circ \Pmh|.
\end{align*}
So, to obtain \eqref{eq:rhev}, it remains only to estimate $\Rm^{\#}\circ \Pmh$.

Using $*$ to denote the adjoint of an element in $\End(\tf)$,
we may write 
\[
\Rm = (\Pmh + \Pmt) \circ \Rm =  \Pmh \circ \Rm + \Pmt \circ \Rm = \Rmh^* + \Rmt^*.
\]
By the symmetry and bilinearity of the $\#$ pairing we then see that
\begin{equation}\label{eq:sharpex}
 \Rm^{\#} \circ \Pmh = (\Rm\#\Rm) \circ \Pmh = (\Rmh^*\#\Rmh^* + 2\Rmt^*\#\Rmh^*)\circ \Pmh + (\Rmt^* \#\Rmt^*) \circ \Pmh.
\end{equation}
But $(\Rmt^*\#\Rmt^*)\circ \Pmh = 0$ since, by \eqref{eq:larelations}, $\left\langle [\Pmt\Rm\varphi_i, \Pmt
\Rm\varphi_j], \Pmh\varphi_k\right\rangle = 0$ for all $i$, $j$, $k$. So
\[
 |(\Rm\#\Rm) \circ \Pmh| \leq C(|\Rmh^*| + |\Rmt^*|)|\Rmh^*| \leq C|\Rm||\Rmh|
\]
by \eqref{eq:sharpex}, and \eqref{eq:rhev} follows.

\subsection{Preservation of K\"ahlerity from the vanishing of $\Am$ and $\Rmh$}

We have seen that if the initial time-slice $(M, g(0))$ is K\"ahler, then $\Am$, $\Bm$, $\Rmh$, and $\Smh$ vanish on this slice.
In the next section, we will show that the backward uniqueness theorem proven in \cite{KotschwarWangConical}
implies that $\Am$, $\Bm$, $\Rmh$, and $\Smh$
must vanish identically on $M\times[0, T]$ for the specific class of solutions $(M, g(\tau))$ we will encounter in the proof of Theorem \ref{thm:kaehlerity}. 
For now we observe that 
the vanishing of these sections indeed imply that $(M, g(\tau))$ will remain K\"ahler relative to the fixed complex structure $J(0) = J_0$.

\begin{lemma}\label{lem:jext}
Let $g$, $J$, $\Rmh$, and $\Pmh$ be as above. If $\Rmh = 0$ and $\nabla \Pmh = 0$ on $M\times [0, T]$,
and $\nabla J = 0$ on $M\times\{0\}$, then $\nabla J = 0$ and $\pdtau J = 0$ on $M\times[0, T]$.
\end{lemma}
\begin{proof}
First observe  that our assumptions imply that $\Sm \circ \Fm = 0$. Indeed, 
\[
\Sm \circ \Fm = \Sm \circ (\Pmh \circ \Fm) = \Smh \circ \Fm = (\nabla \Rm \circ \Pmh) \circ \Fm = \nabla \Rmh \circ \Fm = 0,
\]
using Lemma \ref{lem:projprop}. The observation
that $J$ remains parallel then follows from a direct pointwise calculation.
 Fix $p\in M$ and let $W_{ib}^a = \nabla_iJ^b_a$. Since $D_{\tau} J = 0$,
 we have
 \begin{align*}
  \pdtau W_{ib}^a & = \nabla_i\left(R^a_cJ^c_b - R^c_bJ^a_c\right) - \pdtau\Gamma_{ib}^cJ^a_c + \pdtau\Gamma_{ic}^aJ^c_b\\
  &= J^a_cg^{cd}\left(\nabla_dR_{bi} - \nabla_bR_{di}\right)
     + J^c_bg^{ad}\left(\nabla_c R_{di} - \nabla_dR_{ci}\right)
     - R^a_cW_{ib}^c + R_b^cW_{ic}^a\\
  &= -g^{ac}g^{lm}\nabla_lR_{mibd}J^d_c + g^{ad}g^{lm}\nabla_lR_{midc}J^c_b
    - R^a_cW_{ib}^c + R_b^cW_{ic}^a\\
  &= g^{ad}g^{lm}\left(\nabla_lR_{micb}J^c_d + \nabla_lR_{midc}J^c_b\right)
  - R^a_cW_{ib}^c + R_b^cW_{ic}^a\\
  &= \frac{1}{2}g^{ad}g^{lm}(\Sm \circ \Fm)_{lmibd}- R^a_cW_{ib}^c + R_b^cW_{ic}^a\\
  &= - R^a_cW_{ib}^c + R_b^cW_{ic}^a
 \end{align*}
so that the components of $W$ in the fiber at $p$ satisfy a linear system of ordinary differential equations.
Since $W(p, 0) = 0$, it follows that $(\nabla J)(p, \tau) = W(p, \tau) = 0$ for all $\tau$.

But, if $J$ is parallel and $g$ is Hermitian relative to $J$, we then have $\Rc\circ J = J\circ \Rc$ as endomorphisms of $TM$. So
\[
  \pdtau \Rc = J\circ \Rc - \Rc\circ J = 0,
\]
and $J(\cdot, \tau) = J(\cdot, 0)$ on $M\times [0, T]$.
\end{proof}

More generally, if $g(\tau)$ is a solution to the backward Ricci flow for which
the reduced holonomy remains fixed (e.g., if $(M, g(\tau))$ is complete and of bounded curvature) then any tensor $V_0$ which is parallel on $M$ with respect
to $g(0)$ can be extended to a smooth family $V(\tau)$ of $g(\tau)$-parallel tensors on $M\times[0, T]$ via $D_{\tau} V = 0$ and $V(0) = V_0$.

\section{K\"ahlerity near spatial infinity}
Now we are ready to prove the first assertion in Theorem \ref{thm:kaehlerity}.  Our strategy is fundamentally the same as in \cite{KotschwarWangConical}: from
the asymptotically conical soliton, we construct a self-similar solution to the backward Ricci flow defined for $\tau\in (0, 1]$ on a sufficiently distant end,
which (after adjustment by 
a suitable diffeomorphism) converges smoothly as $\tau\searrow 0$ to the conical (K\"ahler) metric on some neighborhood of infinity. This transforms the essentially elliptic problem of unique continuation at infinity we are initially given into a parabolic problem of backward uniqueness on a finite time interval.

As in the introduction, let $(\Sigma, g_{\Sigma})$ denote  a compact Riemannian manifold of dimension $n-1$,
and let $g_c = dr^2 + r^2g_{\Sigma}$. Denote by $r_c: E_0 \to\RR$ the radial distance $r_c(r, \sigma) = r$ relative to $g_c$.

\begin{proposition}[Proposition 2.1, \cite{KotschwarWangConical}]
\label{prop:brf}
Suppose $(M, \bar{g}, \bar{f})$ is a shrinking Ricci 
soliton asymptotic to the regular cone $(E_0, g_c)$ along the end $V\subset M$. 
Then there exist $K_0$, $N_0$, and $\rad_0> 0$, and a smooth family of maps $\Psi_\tau: E_{\rad_0}\to V$
defined for $\tau\in (0,1]$ satisfying:
\begin{enumerate}
 \item[(1)] For each $\tau\in (0, 1]$, $\Psi_{\tau}$ is a diffeomorphism onto its image and
$\Psi_{\tau}(E_{\rad_0})$ is an end of $V$.
\item[(2)] The family of metrics $g(x, \tau) \dfn\tau\Psi_\tau^*\bar{g}(x)$ 
is a solution to the backwards Ricci flow \eqref{eq:brf}
for $\tau \in (0, 1]$, and converges smoothly as $\tau\searrow 0$ to $g(x, 0) \equiv g_c(x)$ on $E_{\rad_0}$.

\item[(3)] For all $m = 0, 1, 2, \ldots $,
\begin{align}
\label{eq:curvdecay}
\sup_{E_{\rad_0}\times [0,1]} \left(r_c^{m+2}+1\right)\abs{\nabla^{(m)}R} & \le K_0.
\end{align}
Here $|\cdot| = |\cdot|_{g(\tau)}$ and $\nabla = \nabla_{g(\tau)}$ denote the norm and the Levi-Civita connection associated to the metric $g = g(\tau)$.

\item[(4)] If $f$ is the function on $E_{\rad_0}\times (0,1]$ defined by 
$f(\tau)=\Psi_\tau^\ast \bar{f}$, then $\tau f$ extends
to a smooth function on all of $\Ec_{\rad_0}^1$ and there $g$ and $\tau f$ together satisfy
\begin{align}
\label{eq:fid0}
&  \lim_{\tau\searrow 0} 4\tau f(x, \tau) =r_c^2(x),\quad r_c^2-\frac{N_0}{r_c^{2}} \le 4 \tau f \le r_c^2 + \frac{N_0}{r_c^{2}}, 
\end{align}
and
\begin{align}
\label{eq:fid1}
& \pdtau (\tau f) =\tau S,  \  \tau^2|\nabla f|^2 - \tau f = -\tau^2 S, \ \tau \Rc(g) + \tau \nabla\nabla f= \frac{g}{2}.
\end{align}
\end{enumerate}
Here $S$ denotes the scalar curvature of $g$.
\end{proposition}

We apply the above proposition to a shrinker $(M, \bar{g}, \bar{f})$ asymptotic to a K\"ahler cone $(E_0, g_c)$ along the end $V\subset M$
as in Theorem \ref{thm:kaehlerity}, and assume for the rest of the section that $g = g(\tau)$ is the smooth family of metrics on $E_{R_0}\times [0, 1]$
it provides.  Let $J_c$ denote the complex structure associated to $(E_0, g_c)$, and let $J= J(\tau)$ denote the family of almost-complex structures
obtained as solutions to the fiberwise-ODE \eqref{eq:jev}.  Then define $\Pmh$, $\Am$, $\Bm$, $\Rmh$, and $\Smh$ as above, in terms of the solution $g(\tau)$ and let
\begin{align*}
 \ve{X}(\tau) &= \Rmh(\tau)\oplus \Smh(\tau) \in V \oplus (T^*M\otimes V),\\
 \ve{Y}(\tau) &= \Am(\tau) \oplus \Bm(\tau) \in   (T^*M\otimes V)\oplus (T^2(T^*M)\otimes V)
\end{align*}
where $V = \End(\tf)$. The families $\ve{X}$ and $\ve{Y}$ of sections are smoothly defined on all of $E_{R_0}\times[0, 1]$ and vanish identically on $E_{R_0}\times \{0\}$.

Combining \eqref{eq:curvdecay} with Proposition \ref{prop:sysev}, we see that 
\begin{align}\label{eq:xysys}
\begin{split}
    \left|\pd{\ve{X}}{\tau} + \Delta\ve{X}\right| &\leq \frac{C}{r^2_c}\left(|\ve{X}| + |\ve{Y}|\right)\\
    \left|\pd{\ve{Y}}{\tau}\right| &\leq C\left(|\ve{X}| + |\nabla \ve{X}|\right) + \frac{C}{r_c^2} |\ve{Y}|
\end{split}
\end{align}
for some $C = C(n, K_0)$
on $E_{R_0}\times [0, 1]$. Here $|\cdot| = |\cdot|_{g(\tau)}$, $\nabla = \nabla_{g(\tau)}$, and $\Delta =\Delta_{g(\tau)}$. Moreover, we also have
\[
\sup_{E_{R_0}\times[0, 1]}\{|\ve{X}| + |\nabla \ve{X}| + |\ve{Y}|\} \leq C(n, K_0).
\]
The boundedness of the components of $\ve{X}$ follows directly from \eqref{eq:curvdecay}. For the remaining components, note first that we can
bound $\pdtau\nabla \Pmh$ via \eqref{eq:curvdecay},
and hence also $\Am = \nabla \Pmh$.
With \eqref{eq:curvdecay}, we can then bound the components of $\nabla\ve{X}$. Similarly, one can estimate $\Bm = \nabla\nabla \Pmh$ 
via the equation for $\pdtau \nabla\nabla \Pmh$ and the bound on $\nabla\Pmh$.

The vanishing of $\ve{X}$ and $\ve{Y}$ near infinity is now a consequence of the Carleman
estimates established in Propositions 4.7, 5.7, and 5.9 of \cite{KotschwarWangConical}.
\begin{theorem}[\cite{KotschwarWangConical}]\label{thm:bu}
 Suppose $(M, g(\tau))$ is a self-similar solution to \eqref{eq:brf} on $E_{R_0}\times (0, 1]$ with potential $f$ which
 satisfies the conclusions of Proposition \ref{prop:brf}
 relative to the conical metric $g_c$ and the parameters $K_0$ and $N_0$. Let $\ve{X}$, $\ve{Y}$ be smooth, uniformly bounded families of 
 sections of tensor bundles over $E_{R_0}\times [0, 1]$ with the property that, for any $\epsilon > 0$, there is an $R_1 = R_1(\epsilon) \geq R_0$  such that
\begin{align}\label{eq:xysys2}
\begin{split}
    \left|\pd{\ve{X}}{\tau} + \Delta\ve{X}\right| &\leq \epsilon\left(|\ve{X}| + |\ve{Y}|\right)\\
    \left|\pd{\ve{Y}}{\tau}\right| &\leq C_0\left(|\ve{X}| + |\nabla \ve{X}|\right) + \epsilon|\ve{Y}|
\end{split}
\end{align}
for some constant $C_0 > 0$ on $E_{R_1}\times [0, 1]$.  Then, if $\ve{X}$ and $\ve{Y}$ vanish identically on $E_{R_0}\times\{0\}$,
they vanish identically on $E_{R_2}\times [0, \tau_0]$ for some $R_2 \geq R_1$ and $\tau_0 \in (0, 1)$.
\end{theorem}

The two sets of Carleman estimates given in Proposition 4.9 and Propositions 5.7 and 5.9 in \cite{KotschwarWangConical} are valid on any asymptotically
conical shrinking self-similar family of background solutions $g(\tau)$, and the argument in Section 6 of that reference can be applied directly
to our setting to combine them to prove the vanishing of $\ve{X}$ and $\ve{Y}$. Although the particular components of the system $(\ve{X}, \ve{Y})$ 
in \cite{KotschwarWangConical} -- the same as that in \cite{KotschwarRFBU, KotschwarFrequency} --
differ from those of the system considered here (and in fact are sections of different bundles), the structural assumptions
in Theorem \ref{thm:bu} are the only properties of the system used in the application of the Carleman inequalities. But for the labeling of some constants, the argument which follows
equations (6.1) and (6.2) in \cite{KotschwarWangConical} can be used here without change.

Once we know that $\Rmh$ and $\nabla \Pmh$ vanish on $E_{R_2}\times[0, \tau_0]$ for some $R_2$ and $\tau_0 > 0$, we can apply Lemma \ref{lem:jext} to conclude
that $g(\tau)$ is K\"ahler with respect to $J_c$ on the same set. Fix $s\in (0, \tau_0]$ and let $W = \Psi_s(E_{R_2})$, and $\Phi = \Psi^{-1}_s: W\to E_{R_2}$,
where $\Psi_s: E_{R_2}\to V$ is the map from Proposition \ref{prop:brf}. Then $\bar{g}|_W = s^{-1}\Phi^*(g(s))$ is K\"ahler relative to $J = \Phi^*J_c$,
and $\Phi$ is the desired biholomorphism.

\section{Global K\"ahlerity in the complete case.}

To finish the proof of Theorem \ref{thm:kaehlerity}, we now argue that the K\"ahler structure defined on the end $W\subset M$ must
extend to all of $M$ when $(M, g)$ is complete. The real-analyticity of Ricci solitons provides almost all that we need. 

\begin{lemma}\label{lem:genext}
  Suppose $(M, g)$ is a complete connected real-analytic manifold, $E\subset M$ is open and connected,
  and $p\in E$.
  If $\pi_1(E, p) \hookrightarrow \pi_1(M, p)$ is surjective, then any parallel complex structure $J_E$ on $E$ may be extended uniquely
  to a parallel complex structure $J$ on $M$. 
\end{lemma}
\begin{proof}
Let $\hat{J} = J_E(p)$.  The uniqueness of any extension $J$ of $J_{E}$ to $M$ is clear, since, given $q\in M$
we must have
\begin{equation}\label{eq:jext}
J(q) = P_{\gamma}\hat{J}P_{\gamma}^{-1}
\end{equation}
where $P_\gamma:T_pM \to T_qM$
is parallel translation along any piecewise smooth path $\gamma$ from $p$ to $q$. On the other hand, we can also use \eqref{eq:jext} to define $J(q):T_qM\to T_qM$. This will yield a smooth parallel complex structure on all of $M$
provided the extension can be shown to be independent of the path $\gamma$. 

Suppose $\sigma$ is any other path
from $p$ to $q$. Under our assumption on $\pi_1(E)$, the loop $\gamma \cdot \bar{\sigma}$, where $\bar{\sigma}$ 
denotes the reverse parametrization of $\sigma$ from $q$ to $p$ along $\sigma$, 
is homotopic to some loop $\alpha$ contained entirely in $E$.
Then
\[
 P_{\gamma} = P_{\sigma} P_{\bar{\sigma}} P_{\gamma} P_{\bar{\alpha}} P_{\alpha} 
 = P_{\sigma} P_{\beta}P_{\alpha}
\]
where $\beta = \bar{\alpha}\cdot \gamma\cdot \bar{\sigma}$ is a null-homotopic loop based at $p$.

Since $(M, g)$ is real-analytic, the local holonomy of $g$ at $p$ coincides with the reduced holonomy (see, e.g., \cite{Nomizu}), so 
both $\operatorname{Hol}^0_p(M, g)$ and $\operatorname{Hol}_p(E, g)$ leave $\hat{J}$ invariant. Therefore 
$P_{\beta} \hat{J} P_{\beta}^{-1} = P_{\alpha} \hat{J} P_{\alpha}^{-1}= \hat{J}$,
and so $P_{\gamma}\hat{J}P_{\gamma}^{-1} = P_{\sigma}\hat{J}P_{\sigma}^{-1}$
as desired. 
Of course, the restriction of $J_E$ to any path contained in $E$ must coincide with the parallel translation of $\hat{J}$ along that path,
so the connectedness of $E$ implies that $J|_E = J_E$.
\end{proof}

We now give an ad hoc argument to show that the K\"ahler structure extends to the entire manifold. The key ingredient is
the result of Munteanu-Wang \cite{MunteanuWangKaehler} which guarantees
that any complete K\"ahler shrinking gradient Ricci soliton is connected at infinity. We apply this to the universal cover of $M$
to determine that the preimage of any end in $M$ is connected in the universal cover.

\begin{lemma}\label{lem:globallyk}
Suppose $(M, g, f)$ is a complete noncompact gradient Ricci soliton and $g|_E$ is K\"ahler 
on some end $E\subset M$ with complex structure $J_E$. Then $J_E$ extends uniquely to a complex structure $J$ on $M$	
relative to which $g$ is K\"ahler.
\end{lemma}
\begin{proof}
The soliton structure $(M, g, f)$ lifts to  a complete soliton structure $(\tilde{M}, \tilde{g}, \tilde{f})$ on the universal cover $\pi:\tilde{M}\to M$. 
Fix $p\in E$ and a local section 
$\sigma: U\to \tilde{M}$
of $\pi$ over some connected evenly covered neighborhood $U$ of $p$ contained in $E$. Then $J_{\tilde{U}} = (\sigma^{-1})^*J_E$ is a parallel complex structure 
on $\tilde{U} = \sigma(U)$ relative
to which $\tilde{g}|_{\tilde{U}}$ is Hermitian. As Ricci solitons, both $(M, g)$ and $(\tilde{M}, \tilde{g})$ are real-analytic manifolds \cite{IveySolitons}. 
The simple-connectivity of $\tilde{M}$
then implies (as above) that $J_{\tilde{U}}$ extends to a parallel complex structure $\tilde{J}$ on $\tilde{M}$ relative to which $\tilde{g}$ is Hermitian \cite{Nomizu}.
So $(\tilde{M}, \tilde{g}, \tilde{f})$ is a complete K\"ahler gradient Ricci soliton and $\tilde{M}$ must therefore be connected at infinity \cite{MunteanuWangKaehler}. 
But the fundamental group of a complete shrinking soliton is finite \cite{Wylie}, so $\pi:\tilde{M}\to M$ is proper, and $M$ must also be connected at infinity. 
So we may assume that $E = M\setminus D$ for some compact $D\subset M$.

Now we claim that $\tilde{E} = \pi^{-1}(E)$ is connected in $\tilde{M}$. 
Since $\pi$ is proper, $\tilde{D} = \pi^{-1}(D)$
is compact. The set $\tilde{E} = \tilde{M}\setminus \tilde{D}$ therefore must have a unique unbounded connected component $\tilde{C}$. However,
any two of these connected components must be isometric, so every component of $\tilde{E}$ would have to be unbounded as well.
So $\tilde{E} =\tilde{C}$ is connected. This implies that
$\pi_1(E, p)\hookrightarrow \pi_1(M, p)$ is surjective, and the existence of a unique global extension $J$ of $J_E$ follows from Lemma \ref{lem:genext} above.
Since $g$ and $J$ are both parallel, and $g$ is Hermitian relative to $J$ on $E$, it follows that $g$ is Hermitian relative to $J$ everywhere on $M$.
\end{proof}

 The reasoning in Lemmas \ref{lem:genext} and \ref{lem:globallyk} can be applied to other geometric structures which possess similar continuation properties in the real-analytic setting.
 For example, together with the continuation argument from Theorem 6.3, Chapter VI of \cite{KobayashiNomizu}, it can be applied to promote any isometry between the ends
 of two asymptotically conical K\"ahler shrinkers into a global isometry between the shrinkers. With the uniqueness theorem in \cite{KotschwarWangConical}, it follows then, for example, that any two complete noncompact K\"ahler shrinkers
 $(M, g, f)$
 and $(\hat{M}, \hat{g}, \hat{f})$ which are asymptotic to the same cone (along their unique ends) must be \emph{globally} isometric.

\begin{acknowledgement*} The author would like to thank Ben Chow and Will Wylie for useful discussions, and to acknowledge the substantial contribution of Lu Wang
to this work through her joint project \cite{KotschwarWangConical} with the author.
\end{acknowledgement*}


\begin{thebibliography}{99}



 \bibitem[CCZ]{CaoChenZhu} Huai-Dong Cao, Bing-Long Chen, and Xi-Ping Zhu,
 {\it Recent developments on Hamilton's Ricci flow }, 
Surveys in differential geometry, Vol. XII, 47–112, Surv. Differ. Geom., XII, Int. Press,
 Somerville, MA, 2008.

  
\bibitem[CRF]{RFV2P2} Bennett Chow, Sun-Chin Chu, David Glickenstein, Christine Guenther, 
  James Isenberg, Tom Ivey, Dan Knopf, Peng Lu, Feng Luo, and Lei Ni, 
  \textit{The Ricci flow: techniques and applications. Part II. Analytic aspects},
   Mathematical Surveys and Monographs {\bf 144}, American Mathematical Society, Providence, RI, 2008, xxvi+458 pp.

\bibitem[CL]{ChowLu} Bennett Chow and Peng Lu,
\textit{Uniqueness of asymptotic cones of complete noncompact shrinking gradient Ricci solitons with Ricci curvature decay},
C. R. Math. Acad. Sci. Paris {\bf 353} (2015), no. 11, 1007--1009. 

\bibitem[DW]{DancerWang} Andrew Dancer and Mackenzie Wang,
{\it On Ricci solitons of cohomogeneity one},
Ann. Global Anal. Geom. {\bf 39} (2011), no. 3, 259--292.
   


\bibitem[FIK]{FeldmanIlmanenKnopf} Mikhail Feldman, Tom Ilmanen, and Dan Knopf,
{\it Rotationally symmetric shrinking and expanding gradient K\"ahler-Ricci solitons},
J. Diff. Geom. {\bf 65} (2003), no. 2, 169--209.

   
   
\bibitem[H1]{HamiltonHarnack} Richard Hamilton, \emph{The Harnack estimate for the Ricci flow},
   \textit{J. Diff. Geom.} {\bf 37} (1993), no. 1, 225--243. 

\bibitem[H2]{HamiltonSingularities} Richard Hamilton,  
  {\it The formation of singularities in Ricci flow},
  Surveys in differential geometry, Vol. II (Cambridge, MA, 1993), 7--136, Int. Press, Cambridge, MA, 1995.


\bibitem[I1]{Ivey3DSoliton} Thomas Ivey,
{\it Ricci solitons on compact three-manifolds},
Diff. Geom. Appl. {\bf 3} (1993), no. 4, 301--307.
   
   
\bibitem[I2]{IveySolitons} Thomas Ivey,
{\it Local existence of Ricci solitons}, Manuscripta Math. {\bf 91} (1996), no. 2, 151--162.  
   
   
\bibitem[KN]{KobayashiNomizu} Shoshichi Kobayashi and Katsumi Nomizu,
   \textit{Foundations of differential geometry. Vol. I.}
   Interscience Publishers, 
   New York-London, 1963 xi+329 pp.
   
\bibitem[K1]{KotschwarRFBU} Brett Kotschwar,
\textit{Backwards uniqueness of the Ricci flow}, Int. Math. Res. Not. (2010), no. 21. 4064--4097. 

\bibitem[K2]{KotschwarRFHolonomy} Brett Kotschwar, 
\textit{Ricci flow and the holonomy group}, J. Reine Angew. Math. \textbf{690} (2014), 131--161.


\bibitem[K3]{KotschwarFrequency} Brett Kotschwar,
\textit{A short proof of backward uniqueness for some geometric evolution equations},
Int. J. Math. {\bf 27} (2016), no. 12, 1650102, 17 pp. 

\bibitem[KW]{KotschwarWangConical} Brett Kotschwar and Lu Wang, 
\textit{Rigidity of asymptotically conical shrinking gradient Ricci solitons}, J. Diff. Geom. \textbf{100} (2015), no. 1, 55--108. 

\bibitem[MW1]{MunteanuWangKaehler} Ovidiu Munteanu and Jiaping Wang, 
\textit{Topology of K\"ahler Ricci solitons},
J. Diff. Geom. {\bf 100} (2015), no. 1, 109--128. 

\bibitem[MW2]{MunteanuWang4D} Ovidiu Munteanu and Jiaping Wang,
\textit{Geometry of shrinking Ricci solitons}, Compos. Math. {\bf 151} (2015), 2273--2300.

\bibitem[MW3]{MunteanuWangConical} Ovidiu Munteanu and Jiaping Wang,
\textit{Conical structure for shrinking Ricci solitons},
{\tt arXiv:1412.4414 [math.DG]}.

\bibitem[MW4]{MunteanuWangGRSStructure} Ovidiu Munteanu and Jiaping Wang,
\textit{Structure at infinity for shrinking Ricci solitons}, 
{\tt arXiv:1606.01861 [math.DG]}.


\bibitem[N]{Naber4D} Aaron Naber,
{\it Noncompact shrinking four solitons with nonnegative curvature}, 
J. Reine Angew. Math. {\bf 645} (2010), 125--153. 



\bibitem[NW]{NiWallach} Lei Ni and Nolan Wallach,
{\it On a classification of gradient shrinking solitons},
Math. Res. Lett. {\bf 15} (2010) no. 5, 941--955.


\bibitem[No]{Nomizu}  Katsumi Nomizu,
\textit{Recent development in the theory of connections and holonomy groups},
Advances in Math. {\bf 1} 1961 fasc. 1, 1--49.



\bibitem[P]{Perelman2} Grisha Perelman,  
{\it Ricci flow with surgery on three-manifolds}, 
{\tt arXiv:math/0303109 [math.DG]}.


\bibitem[PW]{PetersenWylie} Peter Petersen and William Wylie,
{\it On the classification of gradient Ricci solitons},
Geom. Topol. {\bf 14} (2010), no. 4, 2277--2300. 



\bibitem[W]{Wylie}  William Wylie, \emph{Complete shrinking Ricci solitons have finite fundamental group},
Proc. Amer. Math. Soc. {\bf 136} (2008), 1803--1806.


 \bibitem[Y]{Yang} Bo Yang,
 {\it A characterization of noncompact Koiso-type solitons},
 Internat. J. Math. {\bf 23} (2012), no. 5, 1250054, 13 pp. 


\end{thebibliography}
\end{document}